\documentclass[12pt,a4paper,psamsfonts]{amsart}
\usepackage{amssymb,amscd,amsxtra,calc}
\usepackage{cmmib57}

\usepackage{epsfig}
\usepackage[all]{xy}

\theoremstyle{plain}
    \newtheorem{thm}{Theorem}[section]

    \newtheorem{corollary}[thm]{Corollary}
    \newtheorem{lemma}[thm]{Lemma}

    \newtheorem{theorem}[thm]{Theorem}

\theoremstyle{definition}
    \newtheorem{definition}[thm]{Definition}

    \newtheorem{remark}[thm]{Remark}
\theoremstyle{remark}

    \newtheorem{setup}[thm]{}

\newcommand{\C}{\mathbb{C}}
\newcommand{\BCC}{\mathbb{C}}

\newcommand{\Q}{\mathbb{Q}}
\newcommand{\BQQ}{\mathbb{Q}}
\newcommand{\R}{\mathbb{R}}
\newcommand{\BRR}{\mathbb{R}}
\newcommand{\Z}{\mathbb{Z}}
\newcommand{\BZZ}{\mathbb{Z}}

\newcommand{\SO}{\mathcal{O}}

\newcommand{\BK}{\overline{\rm BK}}

\newcommand{\Mov}{\overline{\rm Mov}}
\newcommand{\N}{\operatorname{N}}
\newcommand{\NS}{\operatorname{NS}}

\newcommand{\PE}{\operatorname{PE}}

\newcommand{\Supp}{\operatorname{Supp}}

\newcommand{\Pic}{\operatorname{\operatorname{Pic}}}

\newcommand{\dom}{\mathrm{dom}}

\newcommand{\Bs}{\mathrm{Bs}}

\usepackage{color}

\begin{document}

\title[Zariski F-decomposition and Lagrangian fibration]{
Zariski F-decomposition and Lagrangian fibration on hyperk\"ahler manifolds
}

\author{Daisuke Matsushita}
\address
{
\textsc{Division of Mathematics} \endgraf
\textsc{Graduate School of Science, Hokkaido University, Sapporo, 060-0810, Japan}}
\email{matusita@math.sci.hokudai.ac.jp}

\author{De-Qi Zhang}
\address
{
\textsc{Department of Mathematics} \endgraf
\textsc{National University of Singapore, 10 Lower Kent Ridge Road,
Singapore 119076}}
\email{matzdq@nus.edu.sg}

\begin{abstract}
For a compact hyperk\"ahler manifold $X$,
we show certain Zariski decomposition
for every pseudo-effective $\BRR$-divisor, and give a sufficient
condition for $X$ to be bimeromorphic to a
(holomorphic)
Lagrangian fibration.
We also prove that any sequence of D-flops between projective hyperk\"ahler
manifolds terminates after finitely many steps.
\end{abstract}

\subjclass[2000]{14J40, 32Q15, 14E30, 53C26}
\keywords{hyperk\"ahler manifold, Zariski decomposition, Lagrangian fibration, termination of flops}

\maketitle

\section{Introduction}

It is a classical result that a pseudo-effective divisor on a compact complex surface has
a Zariski decomposition.
When $X$ is a compact complex manifold of dimension $> 2$
and $D$ a pseudo-effective $\R$-divisor on $X$,
we may consider two types of Zariski-decompositions: either in the sense of
Fujita or in the sense of Cutkosky-Kawamata-Moriwaki as defined before Theorem \ref{ThA}.
In general, such decompositions may not exist.
On the other hand, there are also positive results for big divisors
or for toric varieties;
see \cite[Remark 7-3-6, Th.~7-3-7]{KMM} and \cite[Ch II, Remark 1.17]{ZDA}.

In the first part of this part, we show the existence of such decompositions
on a projective hyperk\"ahler manifold (cf.~the definition in~\ref{setup1}).

Below is our first main theorem, a special case of the more elaborated Theorem \ref{ThA}.

\begin{theorem}\label{ThA'}
Let $X$ be a projective hyperk\"ahler manifold and $D$ a pseudo-effective $\R$-divisor on $X$.
Then there are a birational map $\sigma_1 : X_1 \dasharrow X$ from a projective hyperk\"ahler manifold
$X_1$ and a birational morphism $\sigma_2 : X_2 \to X$ from a projective manifold $X_2$
such that, for each $k \in \{1, 2\}$, $D_k := \sigma_k^*D$ has a
Zariski-Fujita decomposition
$$D_k = P_k + N_k$$
in the sense of Fujita \cite{Fuj}. Namely, we have:
\begin{itemize}
\item[(i)] the divisor $P_k$ is nef, i.e., $P_k \in \bar{K}(X)$, the closure of the
K\"ahler cone $K(X)$; and
\item[(ii)] the divisor $N_k$ is effective; $F \ge \tau^*N_k$ holds whenever
there are a birational morphism
$\tau : X' \to X_k$ and divisor $F \ge 0$ with $\tau^*D_k - F$ nef.
\end{itemize}
\end{theorem}

We remark that the decomposition $D_k = P_k + N_k$ ($k = 1, 2$) in
Theorem \ref{ThA'} is also in the sense of Cutkosky-Kawamata-Moriwaki;
see \cite[Def. 7-3-2, 7-3-5]{KMM}
or the paragraph before Theorem \ref{ThA} for the definition.

\par \vskip 1pc
Our next main theorem is used in the implication ``Theorem \ref{ThA} $\Rightarrow$ Theorem \ref{ThA'}'',
and is a consequence of
Theorem \ref{main}: the termination of flops between projective hyperk\"ahler manifolds.

\begin{theorem}\label{ThB}
Let $X$ be a projective hyperk\"ahler manifold,
and $P$ an effective $\BRR$-divisor in the closed movable cone $\Mov(X)$ {\rm (cf.~the definition in~\ref{setup1})}.
Then there is a birational map $\tau : X' \dasharrow X$
from a projective hyperk\"ahler manifold $X'$ such that $\tau^*P$ is nef.
\end{theorem}

\par \vskip 1pc
In the second part of this paper, we give sufficient conditions
for the existence of Lagrangian fibration on a compact hyperk\"ahler manifold.
Theorems \ref{ThD} and \ref{ThC} are the main results of this part.

Let $X$ be a compact hyperk\"ahler manifold with the Beauville-Bogomolov form $q(*)$
(cf.~\ref{setup1}).
For a nef line bundle $0 \ne L$ on $X$ with $q(L) = 0$,
the {\it Strominger-Yau-Zaslow hyperk\"ahler
conjecture} claims that, in $\Pic(X) \otimes_{\BZZ} \BQQ$, this $L$
is the pullback of a divisor by
a so called (holomorphic) {\it Lagrangian fibration}.
Namely,
a general fibre $F$ of the fibration has $\dim F = \dim X / 2$, and
the $2$-form $\sigma \in H^0(X, \Omega_X^2)
= \BCC \sigma$ restricts to a trivial $2$-form on $F$. The
$F$ is known to be a complex torus by the Liouville-Arnold theorem.

Some partial solutions to the above conjecture have been
obtained by Fu, Hassett-Tschinkel and Sawon
when $X$ is a Hilbert scheme of certain $K3$ surface (cf.~\cite{Saw} for the references therein),
and by Campana-Oguiso-Peternell \cite{COP} for certain non-algebraic $X$.
In \cite{Mat99}, one of the authors proves the above conjecture for non-big,
semi ample $L$ ($\ne 0$) on projective $X$. Late, he extends the result to cover the case with
nef dimension $n(L) \in \{1, \dots, \dim X - 1\}$.
Here $n(L)$ satisfies
$$\kappa(X, L) \le \nu(L) \le n(L)$$
where $\kappa(X, L)$ is the Iitaka $D$-dimension and $\nu(L)$ is
the numerical $D$-dimension. $n(L)$ is defined in \cite{8aut} as $\dim Y$ for a
dominant rational map $f: X \dasharrow Y$ with connected fibres, satisfying:
(i) $f$ is {\it almost holomorphic}, i.e., some fibres of the restriction $f|\dom(f)$
are compact, (ii)  $L$ is numerically trivial on all compact fibres $F$ of $f$ with
$\dim F = \dim X - \dim Y$, and (iii) for every general point $x \in X$
and every curve $C$ passing through $x$ with $\dim f(C) > 0$, one has $L . C > 0$.

The results below are related to the above conjecture.


\begin{theorem}\label{ThD}
Let $X$ be a projective hyperk\"ahler manifold.
Then the following are equivalent.

\begin{itemize}
\item[(1)]
$X$ is bimeromorphic to a projective hyperk\"ahler manifold $X'$ with a $($holomorphic$)$ Lagrangian fibration.
\item[(2)]
$X$ is bimeromorphic to a projective hyperk\"ahler manifold $X'$ admiting a nef,
non-big $\BQQ$-divisor $L'$ with $\kappa(X', L') \ge \dim X / 2$.
\item[(3)]
There is a $\BQQ$-divisor $L$ with $\dim X/2 \le \kappa(X, L) < \dim X$.
\item[(4)]
There is a dominant meromorphic map $g : X \dasharrow B'$ with general fibre of non-general type
and $\dim X/2 \le \dim B' < \dim X$.
\end{itemize}
\end{theorem}

\begin{remark}
{\rm
Even we replace the condition ``$X$ being a projective hyperk\"ahler manifold" in Theorem \ref{ThD} by
a weaker condition ``$X$ being a compact hyperk\"ahler manifold", the same proof works,
except the implication ``$(3) \Rightarrow (1)$" for which we need the compact K\"ahler version of
Theorem \ref{ThB}.
}
\end{remark}

\par \vskip 1pc
The hypothesis in Theorem \ref{ThC} below is weaker than
the termination and abundance conjectures for log pairs
which seem to be harder than the existence of minimal models
(cf.~\cite[Th.~1.2]{BCHM}).
Once we know Theorem \ref{ThB}, the result below (especially the
part (3) $\Rightarrow$ (1)) then follows
as remarked immediately after \cite[Lemma 3.5]{AC}; see also \cite[Th.~3.6]{AC}.

\begin{theorem}\label{ThC}
Let $X$ be a projective hyperk\"ahler manifold.
Assume either $\dim X = 4$, or $($for `$(3) \Rightarrow (1)$'$)$
the existence of
good minimal models
for varieties of
Kodaira dimension zero {\rm (cf.~\cite[Def.~3.50]{KM}, \cite[p.4]{Kaw})}.
Then the following are equivalent.

\begin{itemize}
\item[(1)]
$X$ is birational to a projective hyperk\"ahler manifold $X'$ with a $($holomorphic$)$ Lagrangian fibration.
\item[(2)]
$X$ is birational to a projective hyperk\"ahler manifold $X'$ admiting a nef,
non-big $\BQQ$-divisor $L'$ with the Iitaka $D$-dimension $\kappa(X', L') \ge 1$.
\item[(3)]
There is a $\BQQ$-divisor $L$ with $1 \le \kappa(X, L) < \dim X$.
\item[(4)]
There is a dominant rational map $g : X \dasharrow B'$ with general fibre of non-general type
and $1 \le \dim B' < \dim X$.
\end{itemize}
\end{theorem}

\noindent
{\bf Acknowledgement.}
The authors would like to thank
O.~Fujino, M.~Kawakita, N. Nakayama and H.~Takagi for many suggestions,
and the referee for the suggestions to improve the paper, correcting an error
and encouraging us to generalize Theorem \ref{main} from the dlt case to the current lc case.
The first author is partially supported by Grand-in-Aid \# 18684001,
Japan Society for Promortion of Sciences. The second author
is partially supported by an ARF of NUS.

\section{Preliminary results}

\begin{setup}\label{setup1}
{\rm
{\bf Conventions and terminology}

We use the conventions in \cite{KM}, \cite{GHJ} and Hartshorne's book.

Let $X$ be a compact complex K\"ahler manifold.
For an $\BRR$-divisor $\Delta$ on $X$, we set
$H^0(X, \Delta) := H^0(X, \lfloor \Delta \rfloor)$,
with $\lfloor \Delta \rfloor$ the integral part (or round down) of $\Delta$.

Set $H^{1,1}(X, \BRR) := H^{1,1}(X, \C) \cap H^2(X, \R)$.
The (closed) {\it nef cone} $\bar{K}(X)$ in $H^{1,1}(X, \BRR)$ is
the closure of the {\it K\"ahler cone} $K(X)$ of $X$.
An element in $\bar{K}(X)$ is called a {\it nef class}.
Let $\NS(X)$ be the {\it N\'eron-Severi group}.
Set $\NS_{\Q}(X) := \NS(X) \otimes_{\Z} \Q$ and
$\N^1(X) := \NS(X) \otimes_{\BZZ} \BRR$.
We have $\N^1(X) \subseteq H^{1,1}(X, \R)$.

The (closed) {\it pseudo-effective divisor cone} $\PE(X)$ in $\N^1(X)$
is the closure of effective $\BRR$-divisor classes on $X$.
The {\it closed movable cone} $\Mov(X)$ in
$\N^1(X)$
is generated by the classes of fixed-component free divisors.
We have $\Mov(X) \subseteq \PE(X)$.

We now briefly recall some definitions related to hyperk\"ahler manifolds,
and refer to \cite[pages 171, 176, 182-184, 223-224]{GHJ}
for details. The compact complex K\"ahler manifold $X$ is
called {\it hyperk\"ahler} (or irreducible holomorphic symplectic)
if it is simply connected such that
$H^0(X, \Omega_X^2)$ is spanned by an everywhere non-degenerate two-form $\sigma$.
It follows then $\dim X = 2n$ for some integer $n \ge 1$.
We normalize $\sigma$ so that
$\int_X (\sigma \bar{\sigma})^n = 1$.

There exists a primitive integral quadratic form $q_X(*)$ on $H^2(X, \Z)$, the
{\it Beauville-Bogomolov form}.
Indeed, there is a positive constant $a$ such that
$$q(L) = a \int_X \, L^2(\sigma \bar{\sigma})^{n-1}, \hskip 1pc L \in H^{1,1}(X, \C) .$$
This $q_X(*)$ is non-degenerate of signature $(3, b_2(X) - 3)$.
There is a so called Beauville-Fujiki number $c > 0$ such that
$$q(L)^n =  c L^{2n}, \hskip 1pc L \in H^2(X, \BCC) .$$
The {\it positive cone} $C(X)$ in $H^{1,1}(X, \BRR)$
is the connected component of the open cone
$\{\alpha \in H^{1,1}(X, R) \, | \, q(\alpha) > 0 \}$ that contains a K\"ahler class of $X$.
The closure of $C(X)$ in $H^{1,1}(X, \R)$ is denoted by
$\bar{C}(X)$.
The {\it birational K\"ahler cone} ${\rm BK(X)}$ in $H^{1,1}(X, \R)$ is
$${\rm BK}(X) = \cup_{f: X \dasharrow X'} \, f^*K(X')$$
where $f : X \dasharrow X'$ runs through all bimeromorphic maps $X \dasharrow X'$
from $X$ to another compact hyperk\"ahler manifold $X'$.
The closure of ${\rm BK}(X)$ in $H^{1,1}(X, \R)$ is denoted by
$\BK(X)$. It is known that ${\rm BK}(X) \subseteq C(X)$
and hence $\BK(X) \subseteq \bar{C}(X)$.

For our compact hyperk\"ahler $X$, it is known that
$\Mov(X) = \BK(X) \cap \N^1(X)$;
see also \cite[Ch III, Def.~1.13 and 3.9]{ZDA} and Remark \ref{rThA} (3) below.
Further, $\BK(X)$ coincides with
Boucksom's {\it modified nef cone} ${\rm MN}(X)$ (cf.~\cite[Prop.~4.4 and Lemma 4.9]{Bou} and
\cite[Prop.~28.7]{GHJ}).
}
\end{setup}

\begin{lemma}\label{key}
Let $X$ be a compact hyperk\"ahler manifold with $q(*)$ the primitive Beauville-Bogomolov quadratic form
$($and $q(*, *)$ its
bilinear form$)$ on $H^2(X, \BZZ)$.
Then we have:
\begin{itemize}
\item[(1)]
The birational K\"ahler cone $\BK(X)$ is the intersection of
$\bar{C}(X)$ and the dual of the pseudo-effective divisor $($closed$)$
cone $\PE(X) \subset H^{1,1}(X, \BRR)$
with respect to $q(*, *)$.
If $D \in \PE(X)$ and $q(D, E) \ge 0$ for every prime divisor $E$,
then $D \in \BK(X)$.
\item[(2)]
If $D_1$, $D_2$ are distinct prime divisors, then
$q(D_1, D_2) \ge 0$. Similarly, $q(D_1, D_3) \ge 0$ when $D_3|D_1$ is a pseudo-effective divisor on $D_1$.
\item[(3)]
Suppose that $E_i$ are prime divisors with negative definite matrix $(q(E_i, E_j))_{i,j}$.
Then $E_i$ are linearly independent in the Neron-Severi group $\NS_{\BQQ}(X)$,
and the Iitaka $D$-dimension $\kappa(X, \, \sum E_i) = 0$.
If $D \in \PE(X)$ and $E = \sum e_i E_i$ such that $q(D - E, E_j) \le 0$ for all $j$,
then $D - E \in \PE(X)$.
\item[(4)]
$L \in H^{1,1}(X, \BRR)$ belongs to $C(X)$ if and only if $q(L) > 0$ and $q(L, \omega) > 0$
for some K\"ahler class $\omega$.
\item[(5)]
If an effective $\BQQ$-divisor $L \in \NS_{\BQQ}(X)$ satisfies $q(L) > 0$,
then $X$ is projective, and $L$ is big, i.e., $L = A + E$
for an ample $\BQQ$-divisor $A$ and an effective $\BQQ$-divisor $E$;
further, $|sL| = |M| + F$ for some integer $s \ge 1$,
with $M$ big, $|M|$ the movable part, and
$F$ the fixed part.
\item[(6)]
If $\sigma: X' \dasharrow X$ is a bimeromorphic map from a compact hyperk\"ahler
manifold $X'$,
then it is isomorphic in codimenion one. Hence
$\sigma^*$ is well defined on $H^2(X, \BCC)$ and compatible with its Hodge structure,
the Beauville Bogomolov quadratic form $q(*)$ and the birational K\"ahler cone.
\item[(7)]
If $0 \ne L \in \bar{C}(X)$ and $D \in \PE(X)$ satisfy
$q(L, D) = 0$, then either $D$ and $L$ are parallel in $H^{1,1}(X, \BRR)$
and $q(D) = 0$, or $q(D) < 0$.
\item[(8)]
If $L_1 \equiv L_2$ $($numerical equivalence$)$ for two $\BQQ$-divisors $L_i$,
then $L_1 \sim_{\BQQ} L_2$.
\end{itemize}
\end{lemma}

\begin{proof}
For (1), see \cite[Prop.~28.7]{GHJ}. Note that $q(D) \ge 0$ in the second part.

For (2),
since $\sigma \bar{\sigma} | D_1 \cap D_2$ is weakly positive, one has
$$q(D_1, D_2) = a \int_{D_1 \cap D_2} (\sigma \bar{\sigma} | D_1 \cap D_2)^{n-1} \ge 0 .$$

For the first part of (3), suppose $E' := \sum a_i E_i \equiv \sum b_j E_j =: E''$ (numerical equivalence)
for some $a_i \ge 0$, $b_j \ge 0$ and $E_i \ne E_j$.
Then $0 \ge q(E') = q(E', E'') \ge 0$ by (2), and hence $E' = 0 = E''$
by the negative-definite assumption.
For the second part, write $D = \lim_{t \to \infty} D(t)$ and $D(t) = \sum d(t)_i E_i + D(t)'$
with $d(t)_i \ge 0$, where $D(t)'$ is an effective divisor and contains no any $E_i$.
Since $D(t)$ has a limit and intersecting with a power of a K\"ahler class,
we see that $d(t)_i$ are bounded and we let $\lim_{t \to \infty} d(t)_i = d_i \ge 0$.
Then
$$0 \ge q(D-E, E_j) = \lim_{t \to \infty} q(D(t) - E, E_j) \ge
\lim_{t \to \infty} q(\sum (d(t)_i - e_i) E_i, E_j) = q(\sum (d_i - e_i) E_i, E_j).$$
Hence $\sum (d_i - e_i)E_i \ge 0$ by Zariski's lemma as in \cite[Lemma 3.2]{Miy}.
Thus
$$D - E = \lim_{t \to \infty} (D(t) - E) \, = \,
\sum (d_i - e_i) E_i + \lim_{t \to \infty} D(t)' \, \in \, \PE(X) .$$

(4) is due to the very definition of the positive cone $C(X)$
(containing the K\"ahler cone $K(X)$).

The first part of (5) follows from \cite[Prop.~26.13]{GHJ}
and the claim in its proof,
while the second part follows from the first part.

(6) is proved in \cite[Prop.~21.6 and 25.14]{GHJ} (cf.~(1)).

For (7), see \cite[Ch IV, (7.2)]{BHPV}) and note: $q(*)$ has signature
$(1, b_2(X) - 3)$ on $H^{1,1}(X, \BRR)$.

(8) is true because the irregularity $h^1(X, \SO_X) = 0$.
\end{proof}

\vskip 1pc
To prove Theorem \ref{ThC}, we need the following results in \cite{AC},
with an alternative argument (and notation there) for [ibid., Lemma 2.4]:
the local sections of $f_*\SO(f^*H+D)$ at the point beneath the fibre $F \subset Y$
can be extended globally after replacing the ample $H$ by its multiple.

\begin{lemma} \label{AC} $(cf.$~\cite[Th.~2.3, the proof of Lemma 3.5]{AC}$)$
Let $X$ be a compact K\"ahler manifold with $K_X \sim 0$ and
$g : X \dasharrow B$ a dominant meromorphic map with $B$ projective
and $\dim B \in \{1, \dots, \dim X - 1\}$.
Let $\pi : Y \to X$ be a blowup such that the composition $f = g \circ \pi : Y \to B$
is holomorphic. Then for an ample divisor $H$ on $B$, the divisor
$L := \pi_*f^*H$ and a general fibre $F$ of $f$, we have:
\begin{itemize}
\item[(1)]
$\kappa(X, L) = \dim B + \kappa(F)$, where $\kappa(F)$ is the Kodaira dimension of $F$.
\item[(2)]
Suppose that $X$ is projective and $\kappa(F) = 0$.
Suppose further either $\dim X = 4$, or $F$ has a good minimal model
in the sense of Kawamata \cite{Kaw}. If $L$ is nef, then
$g$ is almost holomorphic $($so the nef dimension $n(L) < \dim X)$.
Namely, some fibres of the restriction $g | \dom(g)$
are compact.
\end{itemize}
\end{lemma}

\section{Proof of Theorems and their consequences}

In this section, we prove the results in the introduction as well as
Theorem \ref{ThA} below. Theorem \ref{ThA} implies Theorem \ref{ThA'},
thanks to Theorem \ref{ThB}.

The part (I) below has been given an analytic proof by
Boucksom \cite[\S 4]{Bou}.
The result in \cite{Bou} is broader since Boucksom decomposes every
pseudo-effective classes in $H^{1,1}(X, \BRR)$ (which may not be a divisor class) as
the sum of a modified nef class and the class of an effective exceptional divisor.
The purpose of including the Part (I) here is to
stress that it has also an algebraic constructive proof; to be precise,
we will see that it follows also from the original Zariski's lemmas \cite{Zar}
as in Fujita \cite{Fuj79} for surfaces (cf.~also \cite[Ch I, \S 3.1 $\sim$ 3.7]{Miy}).

The part (II) below shows, under certain condition $(*)$
(always true for projective $X$ by Theorem \ref{ThB}),
the existence of the {\it Zariski decomposition in the sense of Fujita \cite{Fuj}},
i.e., requiring (II) (i-ii) in Theorem \ref{ThA} (II);
and hence is also the {\it Zariski decomposition in the sense of Cutkosky-Kawamata-Moriwaki},
i.e., requiring (i), (iii) in Theorem \ref{ThA} (II) and (ii)': $N_k$ is effective.

Curious readers may try to extend Theorem \ref{ThA} (II) to deal with
the modified nef classes, by using the characterization of Demailly-Paun for nef
classes.

\begin{theorem}\label{ThA}
Let $X$ be a compact hyperk\"ahler manifold and $D \in \PE(X)$
a pseudo-effective divisor.
Then we have:
\begin{itemize}
\item[(I)]
There is the following {\rm Zariski q-decomposition}
$D = P_D +  N_D$
such that:
\begin{itemize}
\item[(i)]
the divisor $P_D \in \BK(X)$;
\item[(ii)]
the divisor $N_D \ge 0$; either $N_D = 0$, or for $\Supp N_D = \cup N_i$,
the Beauville-Bogomolov quadratic matrix
$(q(N_i, N_j))_{i,j}$ is negative definite; and
\item[(iii)]
$q(P_D, N_i) = 0$ for all $i$.
\end{itemize}
The above Zariski q-decomposition is unique.
Moreover, if $D \ge 0$ then $P_D \ge 0$; if $D$ is a $\BQQ$-divisor then so are $P_D$ and $N_D$.
Finally, for all integers $s \ge 1$, we have the natural isomorphism:
$H^0(X, sP_D) \cong H^0(X, sD)$.
\item[(II)]
Suppose the condition $(*)$ that there is a bimeromorphic map $\sigma_1 : X_1 \dasharrow X$
from a compact hyperk\"ahler manifold $X_1$ such that $\sigma_1^*P_D$ is nef
{\rm (i.e.,~$\sigma_1^*P_D \in \bar{K}(X_1)$; see Theorem \ref{ThB})}. Then there is
a bimeromorphic morphism $\sigma_2 : X_2 \to X$ such that for each $k \in \{1, 2\}$
there is a {\rm Zariski-Fujita decomposition}
$($or {\rm Zariski F-decomposition} for short$)$ $D_k = P_k + N_k$ for $D_k := \sigma_k^*D$,
in the sense of Fujita \cite{Fuj}
{\rm (cf.~\cite[Def.~7-3-2, 7-3-5]{KMM})}.
Thus the following hold.
\begin{itemize}
\item[(i)] the divisor $P_k$ is nef, i.e., $P_k \in \bar{K}(X)$, the nef cone;
\item[(ii)] the divisor $N_k$ is effective; $F \ge \tau^*N_k$ holds whenever
there are a bimeromorphic morphism
$\tau : X' \to X_k$ and divisor $F \ge 0$ with $\tau^*D_k - F$ nef;
\item[(iii)] the natural isomorphism:
$H^0(X_k, sP_k) \cong H^0(X_k, sD_k)$, for all integers $s \ge 1$.
\end{itemize}
The above Zariski F-decomposition is unique. $(iii)$ follows from $(i)$ and $(ii)$. Moreover,
if $D \ge 0$ then $P_k \ge 0$; if $D$ is a $\BQQ$-divisor then so are $P_k$ and $N_k$.
\end{itemize}
\end{theorem}

The condition $(*)$ in Theorem \ref{ThA} is always satisfied by
projective $X$, by Theorem \ref{ThB}.

The result below is parallel to the surface case.

\begin{corollary}\label{Ddim}
With the notation in Theorem {\rm \ref{ThA} (I)}, the Iitaka D-dimension
$\kappa(X, D)$ equals $\kappa(X, P_D)$; we have $\kappa(X, D) = \dim X$ $($maximal case$)$
if and only if $q(P_D) > 0$.
\end{corollary}

\begin{remark}\label{rThA}
\begin{itemize}
\item[(1)]
The Zariski decomposition $D_k = P_k + N_k$ in Theorem \ref{ThA} (II) is also in the sense of
Cutkosky-Kawamata-Moriwaki {\rm (cf.~\cite[Def.~7-3-2, 7-3-5]{KMM})}.

\item[(2)]
By the proof, the Zariski-Fujita decomposition in Theorem \ref{ThA} (II) for
$D_1$ on $X_1$, coincides with the Zariski q-decomposition in Theorem \ref{ThA} (I) for $D_1$.

\item[(3)]
The $N_D$ in Theorem \ref{ThA} (I)
is the smallest effective divisor such that $D - N_D \in \BK(X)$:
if $N' \ge 0$ such that $P' := D - N' \in \BK(X)$,
then $N' \ge N_D$. Indeed, $q(N' - N_D, N_i) = q(P_D - P', N_i) = -q(P', N_i) \le 0$
(cf.~Lemma \ref{key} (1))
for every $N_i \le N_D$ and hence $N' - N_D \ge 0$ by Zariski's lemma
as in the proof of \cite[Ch I, 3.2]{Miy}.

\item[(4)]
By the above reasoning and Lemma \ref{key} (2),
the Zariski q-decomposition in Theorem \ref{ThA}(I) (for projective $X$)
coincides with both the
$\sigma$- and $\nu$-decomposition
in Nakayama \cite[Ch III, Def. 1.16 and 3.2]{ZDA}.
\end{itemize}
\end{remark}

\begin{setup}
{\bf Proof of Theorem \ref{ThA}}
\end{setup}
Part (I) is proved in \cite[\S 4]{Bou}, but the intersection-form based
arguments in Zariski \cite{Zar} and Fujita \cite{Fuj79}
(cf.~also Miyanishi
\cite[Ch I, 3.1 $\sim$ 3.7]{Miy})
work well for (I) with almost no change, though we use $q(*)$ and Lemma \ref{key} instead of the intersection form
for surfaces. Nef divisors on a surface correspond to elements in the birational K\"ahler cone $\BK(X)$ (cf.~Lemma \ref{key} (1)).
The only non-intersection based usage of the Riemann-Roch theorem in \cite[Lemma 3.5]{Miy}
may be replaced by Lemma \ref{key} (5) (3).

The sketch of a constructive proof for (I): Let $E_i$ ($1 \le i \le t_1$) be all prime divisors such that
$q(D, E_i) < 0$. Then $(q(E_i, E_j))_{i,j}$ is a negative definite
matrix [ibid. proof of 3.6].
Let $F_1$ be a (non-negative, by Zariski's lemma) combination of $E_i$ such that
$D_1 := D - F_1$ satisfies $q(D_1, E_i) = 0$ ($1 \le i \le t_1$).
Then $D_1$ is pseudo-effective (and effective when $D$ is effective); cf.~ Lemma \ref{key} (3)
and [ibid. 3.3].
Let $E_j$ ($t_1+1 \le j \le t_1+t_2$) be all prime divisors satisfying $q(D_1, E_j) < 0$.
Then $(q(E_i, E_j))_{1 \le i, j \le t_1+t_2}$ is negative definite;
cf.~[ibid. 3.5] and Lemma \ref{key} (7).
Let $F_2$ be a (non-negative) combination of $E_k$ ($1 \le k \le t_1+t_2$) such that
$D_2 := D - F_2$ satisfies $q(D_2, E_k) = 0$.
Then $D_2$ is pseudo-effective (and effective when $D$ is effective);
cf.~Lemma \ref{key} (3) and [ibid. 3.3].
Since $X$ has finite Picard number and by Lemma \ref{key} (3),
this process will terminate at step $r$, and $P_D := D_r$ and $N_P := \sum_{i=1}^r F_i$
satisfy Theorem \ref{ThA} (I); cf.~[ibid. 3.7] and Lemma \ref{key} (1).

Next we prove Theorem \ref{ThA} (II). Let $D = P_D + N_D$ be as in (I).
Set
$$D_1 := \sigma_1^*D, \,\,\,\, P_1 := \sigma_1^*P_D, \,\,\,\, N_1 := \sigma_1^*N_D .$$
Then $D_1 = P_1 + N_1$ is the Zariski q-decomposition for $D_1$,
by the uniqueness in (I) and since $\sigma_1^*$
is compatible with $q(*)$ (cf.~Lemma \ref{key} (6)).
Let $\pi : X_2 \to X_1$ be a blowup such that the
composition $\sigma_2 = \sigma_1 \circ \pi : X_2 \to X$ is holomorphic.
Set
$$D_2 := \sigma_2^*D, \,\,\,\, P_2 := \pi^*P_1, \,\,\,\, N_2 := D_2 - P_2.$$
Note that $P_2$ is nef and $N_2 = \sigma_2^*N_D + E$ for some
$\sigma_2$- (and hence $\pi$-) exceptional divisor $E$, since $\sigma_1$
is isomorphic in codimension one (cf.~Lemma \ref{key} (6)).

We claim that $E \ge 0$ and hence $N_2 \ge 0$. Indeed, $E = \sigma_2^*D - \pi^*P_1 - \sigma_2^*N_D$
and $-E$ is $\sigma_2$-nef. Further, we have
$\sigma_{2*} E = 0$, since $E$ is $\sigma_2$-exceptional.
Hence $E \ge 0$ by the negativity lemma \cite[Lemma 3.39]{KM}.

Now we show that $D_k = P_k + N_k$ ($k = 1, 2$) is the Zariski F-decomposition as in (II).
The condition (II-i) is of course true. For a direct proof of the condition (II-iii),
by the projection formula (for the first and last equalities below),
since $\sigma_1$ is isomorphic in codimenion one, and applying (I) to $D_1$,
for every integer $s \ge 1$, we have:
$$H^0(X_2, sD_2) = H^0(X, sD) = H^0(X_1, sD_1) = H^0(X_1, sP_1) = H^0(X_2, sP_2).$$
To show (II-ii), we consider $D_2$ only (because $D_1$ is similar and easier),
and replacing $\pi$ by a further blowup, we have only to show the assertion $(**)$:
if $P' := D_2 - F$ is nef for an effective $\BRR$-divisor $F$ then
$F \ge N_2$.
Note that
$\sigma_{2*}P' \in \BK(X)$
(cf.~Lemma \ref{key} (1)(2))). So $\sigma_{2*}F \ge N_D$
by applying Remark \ref{rThA} to $\sigma_{2*}(P') = \sigma_{2*}(D_2 - F) = D - \sigma_{2*}F$.
Now $-(F - N_2) = P' - \pi^*P_1$ is $\pi$-nef, and
$$\sigma_{1*}\pi_*(F - N_2) = \sigma_{2*}(F - \sigma_2^*N_D - E) = \sigma_{2*}F - N_D \ge 0$$ so
$\pi_*(F - N_2) \ge 0$. Hence $F - N_2 \ge 0$ by [ibid.]. This proves $(**)$ and also (II-ii).

The uniqueness of the Zariski F-decomposition is due to the condition (II-ii).
The rest of (II) is from the construction and (I). This completes the proof of Theorem \ref{ThA}.

\begin{setup}
{\bf Proof of Corollary \ref{Ddim}}
\end{setup}

By Theorem \ref{ThA} (I), $D = P_D + N_D$ and $\kappa(X, D) = \kappa(X, P_D)$.
If $q(P_D) > 0$ then $P_D$ is big (cf.~Lemma \ref{key} (5)) and hence $\kappa(X, P_D) = \dim X$.
Conversely, suppose that $\kappa(X, D) = \dim X$. Then
$X$ is Moishezon and K\"ahler and hence projective, and
$P_D$ (and also $D$) are big.
So $P_D = A + E$ for an ample $\BRR$-divisor $A$ and an effective $\BRR$-divisor $E$.
By Lemma \ref{key} (1), $q(P_D) \ge q(P_D, A) \ge q(A, A) > 0$. This proves Corollary \ref{Ddim}.

\begin{setup}
{\bf Proof of Theorem \ref{ThB}}
\end{setup}

Replacing $P$ by a small multiple, we may assume that $(X, P)$ is
terminal (cf.~\cite[Cor.~2.35]{KM}).
By the minimal model program (MMP),
\cite[Cor.~1.4.1]{BCHM} and the termination
of $P$-flops on hyperk\"ahler projective manifolds (to be proved in Theorem \ref{main}),
there is a surjective-in-codimension-one birational map
$\sigma: X \dasharrow X'$ such that $(X', P')$ (with $P' := \sigma_{*}P$) is
$\BQQ$-factorial and terminal, and
$K_{X'} + P'$ is nef. Further, $\sigma = \sigma_r \circ \cdots \circ \sigma_1$,
where each $\sigma_i : X_i \dasharrow X_{i+1}$ is either divisorial or a flip.
Let $P_i \subset X_i$ be the image of $P$.
If $\sigma_1 : X = X_1 \to X_2$ is a $(K_{X_1} + P_1)$-negative
divisorial contraction with $E_1$ the exceptional (necessarily irreducible)
divisor, then $\ell_1 . (K_{X_1} + P_1) < 0$ for a general curve $\ell_1$ in a general fibre of
the restriction map $\sigma_1 | E_1$;
on the other hand, one has $\ell_1 . P_1 \ge 0$
since $P_1 \in \BK(X_1) \cap \N^1(X_1) = \Mov(X_1)$,
the closed movable cone;
see \cite[Ch III, Prop.~1.14]{ZDA} and Remark \ref{rThA},
or \cite[Prop.~4.4, Lemma 4.9]{Bou} and Lemma \ref{key}.
This is a contradiction, noting that $K_{X_1} = 0$.
Thus for $i = 1$, the $\sigma_i$ is a $(K_{X_i} + P_i)$-flip and hence a $P_i$-flop
as defined in \ref{definition_of_flop},
so $X_{i+1}$ is a projective hyperk\"ahler manifold by
\cite[Cor.~1]{Nam}, and $P_{i+1} \in \BK(X_{i+1})$
by Lemma \ref{key} (6); this is true for all $i$, by the same reasoning.
Letting $\tau = \sigma^{-1}$,
this proves Theorem \ref{ThB}.

\begin{setup}
{\bf Proof of Theorem \ref{ThC}}
\end{setup}

$(1) \, \Rightarrow \, (2) \, \Rightarrow (3)$ is always true.

For $(3) \, \Rightarrow \, (4)$,
we may assume that
$\Phi_{|L|} : X \dasharrow B'$ is a dominant rational map with
$\dim B' = \kappa(X, L)$.
By Lemma \ref{AC}, $\kappa(X, L) = \dim B' + \kappa(F)$
for a general fibre $F$ of $\Phi_{|L|}$, and hence $\kappa(F) = 0$ and (4) is true
with $g := \Phi_{|L|}$.

For $(4) \Rightarrow (3)$, take a blowup $\pi : Y \to X$ such that
the composition
$f = g \circ \pi : Y \to B'$ is holomorphic. Let
$H$ be an ample $\BQQ$-divisor on $B'$ and $L := \pi_* f^*H \in \BK(X)$.
We use the argument in \cite[Prop.~3.1]{AC} to show the claim
that $L$ is not big.
Indeed, if $L$ is big (so $X$ is Moishezon and K\"ahler
and hence projective) then so is $\pi^*L | F$ for a general fibre $F \subset Y$
of $f$. On the other hand, $\pi^*L = f^*H + E$ with $E$ effective and
$\pi$-exceptional. Since $X$ is smooth (and hence terminal),
$K_Y = \pi^*K_X + E_{\pi} = E_{\pi}$ with $\Supp E_{\pi}$ containing
every $\pi$-exceptional irreducible divisor, so $E_{\pi} \ge \varepsilon E$,
and hence $K_F = E_{\pi} | F \ge \varepsilon E | F = \varepsilon \pi^*L | F$
for some small $\varepsilon > 0$. Thus
$\kappa(F) \ge \kappa(F, \, \varepsilon \pi^*L|F) = \dim F$ by the
bigness of $\pi^*L | F$. This contradicts the assumption on $\kappa(F)$
in (4).
This proves the claim and hence (3).

For $(3) \, \Rightarrow \, (1)$, replacing $L$ by its multiple,
we may assume that $g = \Phi_{|L|} : X \dasharrow B$ is the Iitaka fibration of $L$
with $\dim B = \kappa(X, L)$, as defined in Theorem 10.3 of Iitaka's GTM book
with the birational uniqueness of this fibration given in Theorem 10.6 of the same book.
Write $|L| = |M| + Fix$ with $Fix$ the fixed part and $\kappa(X, M) = \kappa(X, L)$.
Then $M \in \BK(X)$ and $q(M) = 0$ (cf.~Corollary \ref{Ddim}).

{\it We remark that
for the proof of Corollary \ref{nefL}, if $L$ is non-big, and nef or $L \in \BK(X)$,
then $0 = q(L) \ge q(L, M) \ge q(M) = 0$; so $M$ (and hence $Fix$) are parallel
to $L$; thus $Fix = 0$, after replacing $L$ by its multiple; see
Lemma \ref{key} (7)(8).}

Replacing $L$ by $M$, we may assume that
$Fix = 0$.
Let $\pi: Y \to X$ be the resolution of $\Bs|L|$
so that the composition $f = g \circ \pi : Y \to B$ is holomorphic
and $\pi^*L = f^*H + E$ with $H \subset B$ ample and $E$ effective
and $\pi$-exceptional.

By Theorem \ref{ThB},
there is a birational map $\sigma : X' \dasharrow X$ from
a hyperk\"ahler projective manifold $X'$ such that
$\sigma^*L$ is nef. Replacing $Y$ by a further blowup,
we may assume that there is a birational morphism $\pi' : Y \to X'$
such that $\sigma \circ \pi' = \pi$.
Since $\sigma$ is isomorphic in codimension one,
we have $L = \pi_* f^* H = \sigma_* \pi'_* f^*H$, so
the nef divisor $L' := \sigma^*L$ equals $\pi'_* f^*H$.
Replacing $(X, L, \pi, \Phi_{|L|})$ by $(X', L', \pi', \Phi_{|L'|})$,
we may assume that $L$ is already nef.
By Lemma \ref{AC}, $\kappa(F) = 0$ (since $\kappa(X, L) = \dim B$), and
the nef dimension $n(L) < \dim X$.
So $L$ is semi-ample and the Iitaka fibration $g = \Phi_{|L|}$ is a
(holomorphic) Lagrangian fibration (cf.~\cite[Th.~1.5, Remark 1.6]{Mat}).
This proves (1). We have completed the proof of Theorem \ref{ThC}.

\begin{setup}
{\bf Proof of Theorem \ref{ThD}}
\end{setup}

We prove only $(3) \, \Rightarrow \, (1)$,
and may assume that $L$ is nef and use the notation in Theorem \ref{ThC},
because the other implications can be done as in the proof of Theorem \ref{ThC}, though
the Moishezon-ness of $B'$ (cf.
\cite[Th.~2.3]{COP}) is used in applying
\cite[Th.~2.3, the proof of Prop.~3.1]{AC}.

Set $\dim X = 2n$.
By a result of Verbitsky
(cf.~\cite[Prop.~24.1]{GHJ}) and since $q(L) = 0$ (cf.~Cor.~\ref{Ddim}),
we have $L^n \ne 0$ and $L^{n+1} = 0$. So the numerical $D$-dimension
$\nu(L)$ equals $n$. It is known that $\nu(L) \ge \kappa(X, L) \ge n$.
Hence $\nu(L) = \kappa(L) = n$ and $L$ is abundant. Thus $L$ is semi-ample
by a result of Kawamata, Nakayama and Fujino (cf.~\cite[Th.~4.8]{Fujino}).
So, in $\Pic(X) \otimes_{\BZZ} \BQQ$, the $L$
is the pullback of an ample divisor on a projective variety $B$ by
a (holomorphic) Lagrangian fibration,
by the proof of \cite[Addendum, Th.~1]{Mat99}.
This proves (1).
We have completed the proof of Theorem \ref{ThD}.

\par \vskip 1pc

The proofs of Theorems \ref{ThC} and \ref{ThD} actually also imply the
following. We leave the details for the reader to work out.

\begin{corollary}\label{nefL}
Assume the hypothesis in Theorem $\ref{ThC}$ or $\ref{ThD}$. Then we have:
\begin{itemize}
\item[(1)]
In the situation of Theorem $\ref{ThC} (2)$ or $\ref{ThD} (2)$,
in $\Pic(X') \otimes_{\BZZ} \BQQ$, the $L'$
is the pullback of an ample divisor by
a $($holomorphic$)$ Lagrangian fibration on $X'$.
\item[(2)]
In the situation of Theorem $\ref{ThC} (3)$ or $\ref{ThD} (3)$,
we have $\kappa(X, L) = \dim X/2$.
\end{itemize}
\end{corollary}

We remark a non-vanishing result below.
As pointed out by the referee, the condition is a bit artificial.

\begin{remark}\label{Extr}
Suppose that $0 \ne L$ is a $\BQQ$-divisor $($resp. $\BRR$-divisor$)$
on a compact hyperk\"ahler manifold $X$
such that $L \in \BK(X)$, and
$\BRR_{\ge 0} [L]$ is {\rm not} an extremal ray in the pseudo-effective divisor
$($closed$)$ cone $\PE(X)$ of $X$.
{\it Then we claim that $\Pic(X) \, \ni \, eL \, \sim \, M$
for some effective integral divisor $M$ and
some $e \in {\BQQ}_{> 0}$
$($resp. $e \in {\BRR}_{> 0})$.}

Indeed, if $q(L) > 0$, then $L$ is big (cf.~Lemma \ref{key} (5)) and our claim
is true. We may assume that $q(L) = 0$.
By the assumption, $L \equiv D + G$ (numerical equivalence)
for pseudo effective $\BRR$-divisors $D$ and $G$, which are
not parallel to $L$. By Theorem \ref{ThA} (I) and Lemma \ref{key} (1)
$$\begin{aligned}
D = P_D + N_D, \,\, G = P_G + N_G, \,\,
L \equiv (P_D + P_G) + (N_D + N_G) ; \\
0 = q(L) \ge q(L, P_D + P_G) = q(L, P_D) + q(L, P_G) \ge 0 + 0 ;
\end{aligned}$$
so all become equalities. Now Lemma \ref{key} (7) implies that $P_D \equiv bL$ and $P_G = gL$
for some $b, g \ge 0$.
Hence $(1-b-g)L \equiv N_D + N_G \ge 0$ and $1 \ge b+g$.
If $1 = b+g$, then $N_D = 0 = N_G$, and $D$ and $G$ are parallel to $L$, a contradiction.
Thus $b+g < 1$, and $L \equiv \sum a_i N_i$ for some $a_i \ge 0$
and irreducible components $N_i$ of $N_D + N_G$.
Since the combination $L$ of $N_i$ satisfies $q(L) = 0$,
the matrix $(q(N_i, N_j))_{i,j}$ is not negative definite.
Hence some nontrivial integral combination $N' := \sum b_i N_i - \sum c_j N_j$
(with $N_i \ne N_j$; $b_i, c_j \in \BZZ_{\ge 0}$) satisfies $0 \le q(N') \le
q(\sum b_i N_i) + q(\sum c_j N_j)$ (cf.~Lemma \ref{key} (2)).
Thus we may assume that $N' > 0$ and $q(N') \ge 0$.
Since
$0 = q(L) \ge q(L, N_D+N_G) \ge 0$, all become equalities.
In particular, $q(L, N') = 0$. Therefore, $L$ is parallel to $N'$
and our claim follows (cf.~Lemma \ref{key} (7)(8)).

\end{remark}

\newtheorem{prop}{{\sc Proposition}}[section]
\newtheorem{thmapp}{{\sc Theorem}}[section]
\newtheorem{cor}[thm]{{\sc Corollary}}
\newtheorem{defn}{{\sc Definition}}[section]
\newtheorem{prob}[thm]{{\sc Problem}}
\theoremstyle{definition}
\newtheorem{say}[thm]{}

\renewcommand{\labelenumi}{{\rm (\arabic{enumi})}}
\renewcommand{\labelenumii}{{\rm (\alph{enumii})}}
\newcommand{\qqed}{\hspace*{\fill} $\Box$}

%
%
\section{Terminations of flops between projective hyperk\"ahler manifolds}



In this section, we prove that any sequence of D-flops between projective
hyperk\"ahler manifolds terminates after finitely many steps.

\begin{theorem}\label{main}
 Let $X$ be a projective hyperk\"ahler manifold
 and $D$ an effective $\mathbb{R}$-divisor on $X$.
 Assume that $(X,D)$ has at worst
 log canonical singularities.
 Then there exist no
 sequences of $D$-flops which have infinite length.
\end{theorem}

 We start with the definition of $D$-flops.

\begin{definition}\label{definition_of_flop}
 Let $X$ and $X'$ be $\mathbb{Q}$-Gorenstein
 normal varieties.
 A birational map $\phi : X \dasharrow X'$
 is said to be a $D$-{\it flop} if
 there exist a normal variety $Z$,
 projective
 birational morphisms $f : X \to Z$ and $f' : X' \to Z$
 and an effective $\mathbb{R}$-Cartier divisor $D$ on $X$
 which satisfy the following properties:
\begin{enumerate}
 \item The morphisms $f$ and $f'$ are isomorphic in codimension one.
 \item The maps satisfy the following commutative diagram:
$$
 \xymatrix{
 X \ar[dr]_{f} \ar@{-->}[rr]^{\phi}&  & X' \ar[dl]^{f'} \\
 & Z &
 }
$$
 \item The canonical divisors $K_X$ and $K_X'$ are relatively
       numerically trivial.
 \item The pair $(X,D)$ has only
       log canonical singularities
       and $-D$ is $f$-ample.
 \item The proper transform $D'$ of $D$ is $f'$-ample.
 \item The relative Picard numbers $\rho (X/Z)$ and $\rho (X'/Z)$
       are one.
\end{enumerate}
 Let $\phi_i$ be birational maps which satisfies
 the following sequence:
$$
 \xymatrix{
 X := X_1 \ar[dr]_{f_1} \ar@{-->}[rr]^{\phi_1}& &
      X_2 \ar[dr]_{f_2} \ar@{-->}[rr]^{\phi_2} \ar[dl]^{f^+_1} & &
      X_3 \ar[dr]_{f_3} \ar@{-->}[rr]^{\phi_3} \ar[dl]^{f^+_2} & &
      X_4 \ar[dl]^{f^+_3} \dasharrow \cdots \\
     &Z_{1}& &Z_{2}& &Z_{3}&
 }
$$
  This sequence is said to be {\it a sequence of $D$-flops} if
  there exists an effective $\mathbb{R}$-Cartier divisor $D$ on $X$ such that
  $\phi_i$ is $D^{(i)}$-flop, where $D^{(1)}$ is $D$ and
  $D^{(i)}$ is the proper
  transform of $D^{(i-1)}$ by $\phi_{i-1}$.
\end{definition}

Next we define log discrepancies and
  minimal log discrepancies.

\begin{definition}
  Let $X$ be a normal variety
  and $D$ a Weil divisor on $X$ such that
  $K_{X} + D$ is  $\mathbb{R}$-Cartier.
  For a birational morphism $\mu : X' \rightarrow X$
  from a normal variety $X'$ and
  a prime Weil divisor $E'$ on $X'$, we define
  the {\it log discrepancy} $a(E';X,D)$ by
$$
 a(E' ;X,D) := (\mbox{The coefficient of $E'$ in
 $K_{X'} - \mu^{*}(K_{X}+D)$}) + 1.
$$
  For a proper closed subset $W$ of $X$, we define the {\it minimal
  log discrepancy of} $(X,D)$ by
$$
 \mathrm{mld}(W;X,D):=
 \inf_{\mu (E') \subseteq W}a(E';X,D).
$$
\end{definition}

Now we prove Theorem \ref{main}.
We use the same notation as in Theorem \ref{main}.
According to \cite[Theorem]{shokurov},
to prove Theorem \ref{main},
it is enough to show the following two statements.
\begin{enumerate}
 \item For each $i$, the function $p_{i}$ on $X_{i}$ which is defined by
$$
 x \in X_{i} \mapsto \mathrm{mld}(x;X_{i},D^{(i)})
$$
       is lower semi continuous.
 \item  Let $\mathcal{S}$ be the set of
  the minimal log discrepancies defined by
$$
  \mathcal{S} := \bigcup_{{i}}
 {\rm{mld}}(W_{i};X_i, D^{(i)}),
$$
 where $W_i$ is the exceptional locus of $f_i$.
The set $\mathcal{S}$ satisfies the
       ascending chain condition.
\end{enumerate}

 First we prove (1).
 If $X$ is  smooth projective and
 carries a holomorphic symplectic form,
 then $Z_{1}$ is a symplectic variety by
 \cite[Def.~1.1]{beauville_original}.
 Further, $X_i$ has only $\Q$-factorial terminal singularities,
 by the construction of our $D^{(i)}$-flops.
 Thus
 every $X_i$ is smooth by \cite[Cor.~1]{Nam}.
 Then each $p_{i}$ is lower semi continuous
 by
 \cite[Th.~4.4]{yasuda}.

Next we prove (2).
Since all $ X_{i} $ are smooth, $ \mathrm{mld}(W_{i},X_{i},D^{(i)}) \le \dim X_{i}$.
On the other hand,  all pairs $ (X_{i},D^{(i)}) $ still have
only log canonical singularities since each $f_i$ is the contraction
of a $(K_{X_i} + D^{(i)})$-negative extremal ray.
Hence $0 \le  \mathrm{mld}(W_{i};X_{i},D^{(i)}) $.
Moreover, since
all $ X_{i} $ are smooth and the sets of all coefficients of $ D^{(i)} $
are stable,
the set of log discrepancies of $ (X_{i},D^{(i)}) $ is discrete by
\cite[Th.~5.2]{kawakita}.
Therefore $\mathcal{S}$ is a finite set and we are done.
This proves Theorem \ref{main}.

\begin{remark}
 If   $ X $ is a projective symplectic variety which has only quotient
 singularities, it has been announced that there exist no sequences of $ D $-flops which have
 infinite length in a very recent preprint \cite[Cor.~1.4]{1305.1410}.
\end{remark}


\begin{thebibliography}{99}

\bibitem{AC}
E.~Amerik and F.~Campana,
Fibrations m\'eromorphes sur certaines vari\'et\'es \'a fibr\'e canonique triviale,
Pure Appl. \ Math. \ Q. \textbf{4} (2008), no. 2, part 1, 509--545.

\bibitem{BHPV}
W.~Barth, K.~Hulek, C.~Peters and A.~Van de Ven, Compact complex
surfaces,
Springer 2004.

\bibitem{8aut}
T.~Bauer, F.~Campana, T.~Eckl, S.~Kebekus, T.~Peternell, S.~Rams, T.~Szemberg
and L.~Wotzlaw, A reduction map for nef line bundles,
{\it Complex geometry (G\"ottingen, 2000)}, 27--36, Springer, 2002.


 \bibitem{beauville_original} A.~Beauville,
    {\itshape Symplectic singularities\/},
    Invent. Math., {\bfseries 139}, (2000), 541--549.

\bibitem{BCHM}
C.~Birkar, P.~Cascini, C.~D.~Hacon and J.~McKernan,
Existence of minimal models for varieties of log general type,
J. \ Amer. \ Math. Soc. \ \textbf{23} (2010) 405-468.


\bibitem{Bou}
S.~Boucksom, Divisorial Zariski decompositions on compact complex manifolds,
Ann. \ Sci. \ \'ecole Norm. \ Sup. \ (4) \textbf{37} (2004), no. 1, 45--76.



\bibitem{COP}
F.~Campana, K.~Oguiso and T.~Peternell,
Non-algebraic hyperkähler manifolds,
J. \ Differential Geom. \ \textbf{85} (2010), no. 3, 397--424.

 \bibitem{yasuda} L.~Ein, M.~Musta\c{t}\^{a} and T.~Yasuda,
    {\itshape Jet schemes, log discrepancies and inversion of
      adjunction\/},
    Invent. Math., {\bfseries 153}, (2003), 519--535.

\bibitem{Fujino}
O.~Fujino, On Kawamata's theorem,
On Kawamata's theorem, in : Classification of algebraic varieties, 305--315,
EMS Ser. Congr. Rep., Eur. Math. Soc., Z\"urich, 2011.

\bibitem{Fuj79}
T.~Fujita,
On Zariski problem, Proc. \ Japan Acad. \ Ser. \ A Math. \ Sci. \textbf{55} (1979), no. 3, 106--110.

\bibitem{Fuj}
T.~Fujita,
Zariski decomposition and canonical rings of elliptic threefolds,
\ J. \ Math. \ Soc. \ Japan \textbf{38} (1986), no. 1, 19--37.

\bibitem{GHJ}
M.~Gross, D.~Huybrechts and D.~Joyce,
Calabi-Yau manifolds and related geometries,
Lectures from the Summer School held in Nordfjordeid, June 2001,
Universitext, Springer-Verlag, Berlin, 2003.


\bibitem{kawakita}
M.~Kawakita,
Discreteness of log discrepancies over log canonical triplets on a fixed pair,
J. \ Algebraic Geom. \ (to appear), also: arXiv:\textbf{1204.5248}

\bibitem{Kaw}
Y.~Kawamata,
Minimal models and the Kodaira dimension of algebraic fiber spaces,
J.\ Reine Angew.\ Math.\ \textbf{363} (1985), 1--46.

\bibitem{KMM}
Y.~Kawamata, K.~Matsuda and K.~Matsuki,
Introduction to the minimal model problem,
\emph{Algebraic geometry, Sendai, 1985} (T.~Oda ed.),
Adv.\ Stud.\ Pure Math., \textbf{10},
1987, pp.~283--360.

\bibitem{KM}
J.~Koll\'ar and S.~Mori,
Birational geometry of algebraic varieties,
Cambridge Tracts in Math., \textbf{134}.

\bibitem{Mat99}
D.~Matsushita,
On fibre space structures of a projective irreducible symplectic manifold,
Topology \textbf{38} (1999), no. 1, 79--83;
Addendum, Topology \textbf{40} (2001), 431--432.

\bibitem{Mat}
D.~Matsushita,
On nef reductions of projective irreducible symplectic manifolds,
Math. \ Z. \ \textbf{258} (2008), no. 2, 267--270.


\bibitem{Miy}
M.~Miyanishi,
Noncomplete algebraic surfaces, Lecture Notes in Mathematics, \textbf{857}, Springer,
1981.


\bibitem{1305.1410}
Y.~Nakamura,
On semi-continuity problems for minimal log discrepancies, arXiv: \textbf{1305.1410}


\bibitem{ZDA}
N.~Nakayama,
Zariski-decomposition and abundance,
MSJ Memoirs Vol.~\textbf{14}, Math.\ Soc.\ Japan, 2004.

\bibitem{Nam}
Y.~Namikawa,
On deformations of $\Bbb Q$-factorial symplectic varieties,
J. \ Reine Angew. \ Math. \ \textbf{599} (2006), 97--110.

\bibitem{Saw}
J.~Sawon,
Lagrangian fibrations on Hilbert schemes of points on $K3$ surfaces,
J. \ Algebraic Geom. \ \textbf{16} (2007), no. 3, 477--497.

 \bibitem{shokurov} V.V.~Shokurov,
    {\itshape Letters of a Bi-rationalist V\/},
    Proc. Steklov Inst. of Math., {\bfseries 246}, (2004), 315--336.


\bibitem{Zar}
O.~Zariski,
The theorem of Riemann-Roch for high multiples of an effective divisor on an algebraic surface,
Ann. \ of Math. (2) \textbf{76} (1962) 560--615.

\end{thebibliography}
\end{document}